\documentclass[letterpaper, 10 pt, conference]{ieeeconf}  

\IEEEoverridecommandlockouts                              
\makeatletter
\newcounter{ALC@unique}
\makeatother

\let\labelindent\relax

\usepackage[draft=true,clevethm=true]{defaultpackages}
\usepackage{mathdots}

\usepackage{mynotation}
\usepackage{scenariotree}
\pgfplotsset{compat=1.16}

\providecommand{\md}{}  
\renewcommand{\md}{\theta}  
\newcommand{\obsaug}{Y}
\newcommand{\yb}{\mathbf{y}}
\newcommand{\ub}{\mathbf{u}}
\newcommand{\Npos}{\N_+}
\newcommand{\nc}{n_{\text{c}}}

\newcommand{\cat}[1]{[#1]}

\let\blkdiag\relax\DeclareMathOperator{\blkdiag}{diag}

\newif\ifDraft
\newif\ifArxiv

\Drafttrue
\Arxivtrue

\title{\LARGE \bf
Data-driven distributionally robust control of partially observable jump linear systems
}
\author{Mathijs~Schuurmans and Panagiotis~Patrinos
  \thanks{M. Schuurmans and P. Patrinos are with the Department 
  of Electrical Engineering (\textsc{esat-stadius}), KU Leuven, 
  Kasteelpark Arenberg 10, 3001 Leuven, Belgium.
  Email: \texttt{\{mathijs.schuurmans, panos.patrinos\}@esat.kuleuven.be}}
  \thanks{This work was supported by: FWO projects: No. G086318N; No. G086518N; 
  Fonds de la Recherche Scientifique -- FNRS, the Fonds Wetenschappelijk Onderzoek--Vlaanderen under EOS Project No. 30468160 (SeLMA), 
  Research Council KU Leuven C1 project No. C14/18/068 and the 
  Ford--KU Leuven Research Alliance project No. KUL-0075.}
}

\begin{document}
\maketitle
\thispagestyle{empty}
\pagestyle{empty}

\begin{abstract}
We study safe, data-driven control of (Markov) jump linear systems with 
unknown transition probabilities, where both the discrete mode and the continuous state are to be inferred from output measurements.
To this end, we develop a receding horizon estimator which uniquely identifies a sub-sequence of past mode transitions and the corresponding continuous state, allowing for arbitrary switching behavior.
Unlike traditional approaches to mode estimation,
we do not require an offline exhaustive search over mode sequences to determine the size of the observation window, but rather select it online. If the system is weakly mode observable,
the window size will be upper bounded, leading to a finite-memory observer.
We integrate the estimation procedure with a simple distributionally robust controller,
which hedges against misestimations of the transition probabilities due to finite sample sizes.
As additional mode transitions are observed, the used ambiguity sets are updated,
resulting in continual improvements of the control performance. 
The practical applicability of the approach is illustrated on small numerical examples.\end{abstract}

\section{Introduction}
\subsection{Background and Motivation}
Switching systems are dynamical systems described by a real-valued state vector and a discrete-valued operating mode.
For each of the operating modes, the system dynamics and output mapping may be different. 
This is a wide class of systems, finding applications in a myriad of 
different fields \cite{fischer_OptimalSequencebasedLQG_2013,schuurmans_LearningBasedRiskAverseModel_2020a,blair_ContinuousTimeRegulationClass_1975}.
We refer to Markov switching systems as the particular subclass of 
such systems in which the discrete mode switches are governed by 
a Markov chain.

Our main objective is to control Markov switching systems, while 
simultaneously learning the switching probabilities between the modes, 
which are initially assumed to be completely unknown. 
A promising approach, which systematically balances controller 
performance and robustness with respect to 
misestimations of the probability distributions, is called  
distributionally robust control.
Due to its attractive theoretical properties,
it has become popular for an increasing number of control tasks \cite{rahimian_DistributionallyRobustOptimization_2019, 
coppens_DatadrivenDistributionallyRobust_2020,xu_DistributionallyRobustMarkov_2010,
schuurmans_LearningBasedDistributionallyRobust_2020a}. 

When designing controllers for this class of systems,
it is often assumed that the discrete mode is directly measurable \cite{schuurmans_LearningBasedDistributionallyRobust_2020a,schuurmans_LearningBasedRiskAverseModel_2020a,
chow_FrameworkTimeconsistentRiskaverse_2014,patrinos_StochasticModelPredictive_2014}. However, in practice, 
the discrete mode typically needs to be inferred from measurements of the continuous state or from output measurements. In this work,
we will restrict our attention to the particular case of \ac{MJLS} \cite{costa_DiscretetimeMarkovJump_2005}, and we will only 
assume output measurements to be available.

Our goal is to recursively estimate the active mode,
in order to continually learn the transition probabilities of the underlying Markov chain,
and integrate this procedure with an online controller design which leverages this information to gradually improve performance. 
By adopting a distributionally robust framework, we can do so while 
retaining system theoretic guarantees, such as stability. We 
will illustrate this by means of a linear controller design, but 
keeping in mind more advanced applications, involving model predictive 
approaches \cite{schuurmans_LearningBasedRiskAverseModel_2020a}.

\subsection{Related work}

In general, the task of estimating the state of \ac{MJLS}
has been studied extensively under various assumptions. 
In the most general setting, where both the discrete mode and the 
continuous state are to be estimated from (noisy) output 
measurements, it is well-known that the optimal observer 
requires exponentially growing memory, as was first shown in 
\cite{ackerson_StateEstimationSwitching_1970}.
To combat this restriction, several approximate estimators 
have been proposed \cite{ackerson_StateEstimationSwitching_1970,tugnait_DetectionestimationSchemeState_1979}, e.g.,
the \ac{IMM} approach \cite{mazor_InteractingMultipleModel_1998}
has been a popular choice in the case where the transition probabilities are known. More recently, in \cite{alriksson_ObserverSynthesisSwitched_2006}, techniques from 
relaxed dynamic programming were used to design an efficient 
receding horizon estimator which discards unlikely mode sequences, while retaining accuracy bounds.
Since most of these approximate methods rely on the knowledge of the transition probabilities to merge or discard certain mode sequences, they are not directly applicable in the present setting, as no prior knowledge of the transition probabilities is available.

Others have proposed to circumvent the mode estimation problem 
and directly estimate the underlying state, e.g., \cite{bako_NewStateObserver_2011}. However, since our end goal involves learning the underlying transition matrix for prediction purposes, continuous state estimates alone do not suffice for our purposes.

A parallel line of study has focused on the possibility of uniquely 
identifying the mode sequence and initial state from noiseless 
output measurements, leading to different notions
of observability in the context of (linear) switching systems \cite{ji_ControllabilityObservabilityDiscretetime_1988}.
In \cite{vidal_ObservabilityIdentifiabilityJump_2002}, simple rank conditions for the observability of autonomous jump linear systems with unknown switching sequence were given.
These results were later extended to systems with a control input 
\cite{elhamifar_RankTestsObservability_2009}. 
Although these conditions are attractive from a computational point of view, since they do not suffer from exponential growth in complexity, they require a minimal dwell-time assumption on the (unknown) mode sequence. This assumption excludes the case where the underlying mode sequence is generated by a Markov chain, especially if there is no prior information on the transition probabilities. 

Around the same time, Babaali and Egerstedt proved complementary results, 
allowing for arbitrary switching behavior \cite{babaali_PathwiseObservabilityControllability_2003,babaali_ObservabilitySwitchedLinear_2004,babaali_AsymptoticObserversDiscreteTime_2005}. It has been 
shown that in these circumstances, it is often not possible 
to uniquely identify the mode sequence at the current time step. This 
consideration has led to more relaxed notions of mode observability \cite{baglietto_ActiveModeObservability_2007,baglietto_ActiveModeObservation_2009,alessandri_RecedinghorizonEstimationSwitching_2005,baglietto_ModeobservabilityDegreeDiscretetime_2014}, which we will 
utilize in this work.

Roughly speaking, a system is said to be mode observable if there 
exists a finite observation horizon $N$, such that the mode sequence 
(and the initial state) can be recovered from $N$ measurements (see \Cref{sec:mode-observability} for a formal definition). 
In prior observer designs, this horizon was assumed to be determined offline. Although the determination of this horizon was shown to 
be decidable in theory \cite{babaali_PathwiseObservabilityControllability_2003}, the computational complexity explodes rapidly with the required horizon $N$, potentially rendering this offline procedure intractable, even for small-scale systems \cite{babaali_PathwiseObservabilityControllability_2003}.

\subsection{Contributions}
Our first and main contribution is the development of a recursive estimator 
for the mode and the continuous state of Markov linear systems.
As we do not assume any prior knowledge on the transition probabilities,
our scheme permits arbitrary switching behavior. This recursive procedure will automatically select the required window size, alleviating the need to determine the mode-observability index offline.
Additionally, we integrate the proposed mode observer with a data-driven 
distributionally robust controller design and illustrate how potential non-uniqueness of the current mode and current state can be circumvented
using an output-feedback approach.
Finally, we demonstrate the practicality of our approach using illustrative numerical experiments.

\subsection{Notation}
Given $a \leq b \in \N$, we define $\natseq{a}{b} \dfn \{
n \in \N \mid a \leq n \leq b \}$. We define the strictly positive natural numbers by $\Npos \dfn \N \setminus \{0\}$. Let $x = (x_t)_{t \in \N}$ denote some (potentially vector-valued) sequence; given indices $k\geq l \in \N$, $\seq{x}{k}{l} = (x_{k}, \dots, x_{l})$ is the subsequence over the time steps $k$ to $l$.
Given a matrix $A$, $A^{\dagger}$ is its (Moore-Penrose) pseudo-inverse. 
Given matrices $A_i$, $i=0, \dots, n$, we denote their 
product (in the order of the index $i$) by $\prod_{i=0}^{n} A_i = A_0 A_1 \dots A_n$. We denote the cardinality of a (finite) set $\Theta$ by $|\Theta|$. Given symmetric matrices $Q, Q'$, we write $Q \succ Q'$ to signify that $Q-Q'$ is positive definite.
Finally, we denote by $\simplex_{\nModes} \dfn \{ p \in \Re^{\nModes} \mid p \geq 0, \trans{\1}p =1 \}$ the $\nModes$-dimensional probability simplex.

\section{Preliminaries and problem statement}
Let $\{\md_t\}_{t \in \N}$ be a time-homogeneous Markov chain, defined 
on some probability space $(\Omega, \F, \prob)$ and taking 
values on the finite set $\W \dfn \{1,\dots, \nModes\}$.
We denote its (unknown) transition matrix $\transmat = (\transmat_{i,j})_{i,j \in \W} \in \simplex_{\nModes}^{\nModes}$, where $\simplex_{\nModes}^{\nModes}$ denotes the set of $\nModes$-by-$\nModes$ stochastic matrices, that is, $P_{i,j} = \prob[\md_{t}=j \mid \md_{t-1}=i]$. 
We will denote the $i$th row of $\transmat$ by $\row{\transmat}{i} \in \simplex_{\nModes}$. We will assume that the Markov chain is 
$(\md_t)_{t\in\N}$ is ergodic, i.e., there exists a value $k \in \N_{>0}$, such that  $\transmat^k > 0$ element-wise, for some $k \geq 1$.
This assumption ensures that every mode of the chain gets visited infinitely often \cite[Ex. 8.7]{billingsley_ProbabilityMeasure_1995}, effectively allowing us to learn the complete transition matrix from a sample trajectory.

Driven by this Markov chain is the \ac{MJLS}
\begin{subequations} \label{eq:system}
    \begin{align}\label{eq:dynamics}
        x_{t+1} &= A_{\md_t} x_t + B_{\md_t}u_t\\
        y_{t} &= C_{\md_t} x_t, \label{eq:measurement}
    \end{align}
\end{subequations}
where $x_t \in \Re^{\ns}$ is the (continuous) state of the system, $u_t \in \Re^{\na}$ is the input and $y_{t} \in \Re^{\ny}$ is the output 
at time step $t$. 
We refer to the value $\md_t$ as the mode at time $t$.

Given the initial state $x_0$, a sequence of previous control 
actions $\seq{u}{0}{t-1} = (u_0, \dots, u_{t-1}) \in \Re^{t \na}$ and the true mode sequence 
$\md^{\star} = (\md_{0}, \dots, \md_t) \in \W^{t+1}$, the sequence 
of measurements $\seq{y}{0}{t} = (y_{0}, \dots, y_{t}) \in \Re^{\ny(t+1)}$ 
satisfies $\seq{ y}{0}{t} = \obsaug(\md^\star, x_0,\seq{ u}{0}{t-1})$, with 
\begin{equation}  \label{eq:measurement-equation}
    \obsaug(\md^{\star}, x_0, \seq{u}{0}{t-1}) \dfn \obs(\md^\star)x_0 + G(\md^\star)\seq{ u}{0}{t-1}, 
\end{equation}
where
\begin{align*}
    \obs(\md)
    {}\dfn{}
    \smallmat{
        C_{\md_0} \\
        C_{\md_1} A_{\md_0} \\
        \vdots \\ 
        C_{\md_t} \prod_{k=1}^{t} A_{\md_{t-k}}
    }
\end{align*}
is the observability matrix and 
\begin{align} \label{eq:hankel}
    G(\md) &\dfn H(\md) \blkdiag(B_{\md_0}, \dots, B_{\md_{t-1}}), \\
    H(\md) &\dfn \smallmat{
        0                                       & 0             & \dots  & 0    \\
    C_{\md_1}                                   & 0             & \dots  & 0    \\
    C_{\md_2}A_{\md_{1}}                        & C_{\md_2}     & \dots  & 0    \\ 
     \vdots                                     & \vdots        & \ddots & \vdots \\
    C_{\md_t}\prod_{k=1}^{t-1}A_{\md_{k-i}}     & C_{\md_t}\prod_{k=1}^{t-2}A_{\md_{k-i}}         &  \dots   & C_{\md_{t}} 
    },
\end{align}
describe the effects of the controls.
We will now review several standard notions from the literature regarding 
observability of the mode sequence, and adapt them to our purposes where 
needed.

\subsection{Mode observability} \label{sec:mode-observability}
The task of recovering the mode sequence (or \emph{path}) from a sequence of measurements rests on our ability
to distinguish two mode sequences from one another,
which leads to the following concept, originally introduced by \cite{babaali_ObservabilitySwitchedLinear_2004} as \emph{controlled discernibility}.
\begin{definition}[Discernibility] \label{def:discernible}
    We say that a pair of mode sequences $\md, \md' \in \W^{N}$ 
    of length $N$ is discernible with respect to a control sequence $u \in \Re^{(N-1)\na}$ if, 
    \begin{equation} \label{eq:discernible-condition}
        \obsaug(\md, x, u) \neq \obsaug(\md', x', u), \quad \text{ for all }x, x' \in \Re^{\ns}.
    \end{equation}
    Otherwise, we say that the pair is indiscernible with respect to 
    $u$.
\end{definition}
In words, two paths are discernible if there is no pair of initial states for which the two paths would yield the same output measurements.

\begin{remark}
In the context of autonomous systems, i.e., whenever $B_i = 0$ for all $i \in \W$, 
a weaker notion of discernibility is typically considered, 
where a condition akin to \eqref{eq:discernible-condition} is imposed for \textit{almost 
every} $x, x'$ \cite{babaali_ObservabilitySwitchedLinear_2004},
or all $x, x'$, such that $x \neq 0$ or $x'\neq 0$
\cite{alessandri_RecedinghorizonEstimationSwitching_2005}. 
This relaxation is necessary in this case,
since \eqref{eq:discernible-condition} would require the 
column spaces of $\obs(\md)$ and $\obs(\md')$ to be 
fully disjoint,
which is impossible as both trivially contain the zero vector. 
Unfortunately, this weakened discernibility notion is not sufficient to ensure that the state and mode sequence can be recursively determined online. An additional assumption needed for this purpose is known as \emph{backward discernibility}. We refer to \cite{babaali_AsymptoticObserversDiscreteTime_2005,halimi_ModelBasedModesDetection_2015} for more details as we will 
not consider this case here.
\end{remark}

We can now introduce the following notions of mode observability
(\acs{MO}\acused{MO}) for 
system \eqref{eq:system}, following roughly the terminology 
of \cite{baglietto_ActiveModeObservability_2007,alessandri_RecedinghorizonEstimationSwitching_2005}.
\begin{definition} \label{def:alpha-omega-MO}
    Consider $N \in \Npos$ and $\alpha, \omega \in \N$, with 
    $\alpha + \omega < N$. We say that 
    system \eqref{eq:system} is $(N,\alpha,\omega)$-\ac{MO} 
    if for all paths $\md, \md' \in \W^{N+1}$, with 
    $\seq{\md}{\alpha}{\omega} \neq\seq{ \md}{\alpha}{\omega}$, 
    $\md$ and $\md'$ are discernible with respect to 
    \emph{almost every} control sequence $u \in \Re^{\na(N-1)}$, 
    that is, all $u \in \Ufeas_{N}^{\alpha,\omega} \dfn \Re^{\na (N-1)} \setminus \mathcal{K}^{\alpha,\omega}_N$, with $\mathcal{K}_{N}^{\alpha,\omega}$ a set of Lebesgue measure 0.
\end{definition}
We refer to $\Ufeas_N \dfn \bigcup_{\alpha, \omega} \Ufeas_{N}^{\alpha,\omega}$ as the set of \emph{discerning} $N$-step control sequences.
This notion can be characterized as follows \cite{baglietto_ActiveModeObservability_2007,babaali_ObservabilitySwitchedLinear_2004}.
\begin{lem}[Controlled discernibility \cite{babaali_ObservabilitySwitchedLinear_2004}]
\label{lem:controlled-discernibility}
    Two paths $\md, \md' \in \W^{N+1}$ are discernible with respect to almost every control sequence if  
    \begin{equation} \label{eq:discernibility}
        (I - P)(G(\md)-G(\md')) \neq 0,
    \end{equation}
    where $P$ denotes the orthogonal projection onto the column space of $\regmatrix{\obs(\md) & \obs(\md')}$.
\end{lem}
We call any system satisfying the above \emph{weakly} \ac{MO}. 
\begin{definition}[Weak mode observability]
    We say that system \eqref{eq:system}
    is weakly \ac{MO} at index $N$, if there exists a pair $(\alpha, \omega)$, 
    such that it is $(N,\alpha, \omega)$-\ac{MO}.
    We call the lowest such $N$ the index of observability. 
\end{definition}
\begin{assumption} \label{assum:MO}
    System \eqref{eq:system} is weakly \ac{MO} at some (unknown) index 
    $N$.
\end{assumption}
One can argue that weak \ac{MO} is the minimal requirement
to guarantee a priori that any mode observation scheme will produce a meaningful result. If it does not hold, then there is no guarantee that for any initial state, even a single mode can ever be uniquely determined, regardless of the number of output measurements.

In practice, verification of (weak) mode observability may require a rapidly growing number of verifications of \eqref{eq:discernibility}, rendering even offline computation of the observability index $N$ intractable. 
For this reason, we propose a scheme which does not require the 
observation horizon to be selected beforehand.
Rather, it is automatically tuned to the smallest required value. 
We will see that this typically yields a smaller observation window than 
required by mode observability (see \Cref{sec:examples}). 

\section{Online mode estimation}
In this section, we describe a mode estimation procedure
which stores a window of $N_t$ past measurements, where $N_t$
is adaptively selected online.
To ease notation, we will denote by $\yb_t \dfn\seq{ y}{t-N_t}{t}$ the sequence of measurements collected over the observation window at time $t$. Similarly, we will denote by $\ub_t \dfn\seq{ u}{t-1-N_t}{t-1}$ the known control sequence over the same window.
At every time instance, the estimation procedure, summarized in \Cref{alg:mode_estimation}, consists of two main steps, discussed below:
\begin{inlinelist}
    \item update the set of \emph{consistent mode sequences} (\Cref{sec:step1-consistent-mode-sequences}); and
    \item update the horizon length (\Cref{sec:horizon-length-selection}).
\end{inlinelist}

In the remainder of this section, we will make the following assumption.
\begin{assumption}\label{assum:discernible}
the applied control sequences $\ub_t$ are discerning for all $N_t \geq N$,
where $N$ denotes the index of mode-observability of the system.
\end{assumption}
\begin{remark} \label{rem:noise}
This assumption can be enforced by imposing it explicitly as a condition during controller synthesis \cite{baglietto_ActiveModeObservability_2007}. Unfortunately, this leads to a nonconvex constraint, which is highly undesirable in most design procedures for linear systems. However, since by assumption on the underlying system, the set of non-discerning control sequences is null, it suffices in practice to 
add a very small random number (e.g., $e_u \sim \mathcal{N}(0, \epsilon I_{\na})$, with $\epsilon>0$ close to zero) to the control action and thus apply $\tilde u_t = u_t + e_u$ to the system instead. In doing so, the sequence $\ub_t$ will be discerning with probability one. 
\end{remark}
\subsection{Consistent mode sequences} \label{sec:step1-consistent-mode-sequences}
At every time instance $t$,
the procedure keeps track of the set of modes sequences of length $N_t$, defined as follows.
\begin{definition}[Consistent mode sequences]
    For a given measurement sequence $y \in \Re^{(N+1)\ny}$, 
    and control sequence $u \in \Re^{N \na}$, we denote by 
    \begin{equation*} \label{eq:definition-consistency}
        \Theta(y, u) \dfn \{ \md \in \W^{N+1} \mid \exists x \in \Re^{\ns}: y = \obsaug(\md, x, u) \}
    \end{equation*}
    the set of mode sequences consistent with $y$ and $u$.
\end{definition}
\begin{remark} \label{rem:indiscernible}
    Note that by definition, any pair of elements $\md, \md' \in \Theta(y,u)$ is indiscernible with respect to $u$. 
\end{remark}
\begin{remark}
In the following, we will always index a mode sequence $\md \in \Theta(\yb_t, \ub_t)$ relative to the start of the prediction window,
i.e., $\md = (\md_0, \dots, \md_{N_t})$, 
and not relative to the absolute time index $t$,
as in $(\md_{t- N_t}, \md_{t- N_t+1}, \dots, \md_{t})$. 
This should cause no confusion as the final element in the sequence, $\md_{N_t}$, will always correspond to the absolute time $t$. 
\end{remark}
Given a mode sequence $\md$,
one can verify its consistency by solving a single least-squares problem.
More specifically, $\md \in \Theta(y,u)$ if and only if 
\begin{equation} \label{eq:check-consistent} 
    (I-\obs(\md)\obs(\md)^\dagger)\tilde y(\md) = 0,
\end{equation}
with $\tilde{y}(\md) \dfn y - G(\md) u$.
In a naive implementation, constructing the set $\Theta(\yb_t, \ub_t)$
would require $O(\nModes^{N_t+1})$ verifications of \eqref{eq:check-consistent}.
However, as the next result shows, only the extensions 
of mode sequences that are in $\Theta(\yb_{t-1}, \ub_{t-1})$ need to 
be considered. 
This reduces the computational burden to $O(|\Theta(\yb_{t-1}, \ub_{t-1})| \nModes)$, which could potentially entail quite a considerable reduction.

\begin{lem} \label{lem:incremental-update}
Let $\md' \dfn \cat{\md\bar{\md}} \in \W^{N+1}$ denote
the concatenation of a path $\md \in \W^{N}$ with a single mode $\bar{\md} \in \W$, for some $N \in \Npos$.
If $\md' \in \Theta(\seq{y}{1}{N+1},\seq{u}{1}{N})$, then there must exist a mode $\underline{\md} \in \W$
such that $\cat{\underline{\md} \md} \in \Theta(\seq{y}{0}{N},\seq{ u}{0}{N-1})$.
Similarly, if $\cat{\underline \md \md \bar{\md}} \in \Theta(\seq{y}{0}{N+1},\seq{ u}{0}{N})$, then $\cat{\underline \md \md} \in \Theta(\seq{y}{0}{N},\seq{ u}{0}{N-1})$, 
\end{lem}
\begin{proof}
    It follows directly from the definition of $\Theta(\argdot)$. See the \Cref{proof:lem:incremental-update} for details.
\end{proof}

\subsection{Window size selection} \label{sec:horizon-length-selection}
Once the set of consistent mode sequences has been computed,
the next step is to determine whether it is required to increase the window size $N_t$.
To this end, let $\nc \in \Npos$ be a user-specified positive number and define the set  
\begin{equation} \label{eq:indices}
    \begin{aligned}
      I_{t}^{\nc} \dfn 
      \left\{
        k \in \natseq{0}{\bar{N}_t} 
        {}\sep{}
        \begin{matrix}
            \theta_{k+i} = \theta'_{k+i}, \forall \theta, \theta' \in \Theta(\yb_t, \ub_t),\\
            \forall i \in \natseq{0}{\nc-1}
        \end{matrix}
    \right \},
    \end{aligned}
\end{equation}
with $\bar{N}_t \dfn \max\{N_t-\nc, 0\}$,
containing the indices at which all consistent mode sequences are identical for $\nc$ consecutive time steps. If $I_t^{\nc} = \emptyset$, then we set $N_{t+1} = N_{t} + 1$. By default, one would take $\nc=1$. However, in the case where one is interested in observing mode \emph{transitions} rather than simply the modes at potentially non-consecutive time instances (as we are), it is desirable to take $\nc=2$.
This criterion for the selection of $N_t$ is justified by the following observations. 
\begin{lem} \label{prop:extension-of-MO-index}
    If system \eqref{eq:system} is $(N, \alpha, \omega)$-\ac{MO}, then it is also $(N+1, \alpha, \omega)$-MO.
\end{lem}
\begin{proof}
    See \Cref{proof:prop:extension-of-MO-index}.
\end{proof}

\begin{proposition} \label{lem:condition-not-MO}
    If the system is weakly \ac{MO} at index $N$, then 
    there exists an index $k$, such that for all sufficiently large 
    $t$, 
    \begin{equation} \label{eq:test-MO}
        \md_{k+i} = \md_{k+i}'\quad \forall \md, \md' \in \Theta(\yb_t, \ub_t) 
    \end{equation} 
    for all $i \in \natseq{0}{n_c-1}$.
\end{proposition}
\begin{proof}
    This is a direct consequence of \Cref{def:alpha-omega-MO} and the fact that two paths $\theta, \theta' \in \Theta(\yb_t, \ub_t)$
    are indiscernible (see \Cref{rem:indiscernible}). We consider 
    two distinct cases. 
    \begin{proofsteps}
        \item \label{case:large-N} Suppose that $N_t = N$, then weak mode-observability and 
    \Cref{assum:discernible} imply that for some $\alpha, \omega$,
    \[ 
        \seq{\md}{\alpha}{N_t-\omega} = \seq{\md}{\alpha}{N_t-\omega}, \; \forall \md, \md' \in \Theta(\yb_t, \ub_t), 
    \]
    since all paths in $\Theta(\yb_t, \ub_t)$ are indiscernible (see \Cref{rem:indiscernible}). By \Cref{prop:extension-of-MO-index}, this 
    can be extended inductively to any $N_t > N$. Hence, if $N_t \geq 
    \max\{N, n_c + \alpha + \omega - 1\}$, then \eqref{eq:test-MO} holds 
    for all $i \in \natseq{0}{\nc - 1}$.
    \item If, alternatively $N_t < N$, then two situation may occur.
    Either \eqref{eq:test-MO} holds and there is nothing to prove,
    or \eqref{eq:test-MO} does not hold. 
    The latter case implies that $I_{t}^{\nc} = \emptyset$ and therefore 
    $N_{t}$ is increased by one.
    \end{proofsteps}
    This situation may occur a finite number 
    of times,
    until eventually $N_{t} = \max\{N, \nc + \alpha + \omega -1\}$ and we arrive back in case \ref{case:large-N}.
\end{proof}
From \Cref{lem:condition-not-MO} it follows that the window sizes $N_t, t \in \N$ generated by \Cref{alg:mode_estimation} are upper bounded, and furthermore, that asymptotically, an infinite number of consecutive mode sequences of length $\nc$ will be observed. These mode sequences can be obtained at every time step from $\Theta_t$ and $I_{t}^{n_c}$.

\begin{algorithm}[t]
    \caption{Mode estimation}\label{alg:mode_estimation}
    \begin{algorithmic}[1]
        \Require initial measurement $y_0$, system \eqref{eq:system}, $\nc>0$
        \State $t\gets0$, $N_t \gets 0$, $\Theta_t \gets \{\md \in \W \mid \exists x: y_0 = C_{\md}x \}$, 
        \Loop
            \State Compute $I_{t}^{\nc}$ using \eqref{eq:indices}
            \If{$|I_{t}^{\nc}| = 0$} \label{step:MO-check}
                \State $N_{t+1} \gets N_t+1$
            \EndIf
            \State $t \gets t + 1$
            \State Measure: $\yb_{t} \gets\seq{y}{t-N_{t}}{t}, \ub_t \gets\seq{ u}{t-N_{t}-1}{t-1}$ 
            \State $\Theta_{t} \gets $ \textproc{update}($\Theta_{t-1}$, $\yb_t, \ub_t$)
        \EndLoop
        \Procedure{update}{$\Theta, \yb_t, \ub_t$}
            \State $\bar{\Theta} \gets \emptyset$
            \ForAll{$\underline{\md} \in \Theta$}
                \ForAll{$\bar \md \in \W$}
                    \State $\md \gets \cat{\seq{\underline{\md}}{t-N_t}{t} \bar{\md}}$ \Comment{Concatenate paths}
                    \If{\eqref{eq:check-consistent} holds for $y=\yb_t$, $u=\ub_t$}
                        \State $\bar \Theta \gets \bar \Theta \cup \theta$
                    \EndIf
                \EndFor
            \EndFor
            \Return $\bar \Theta$
        \EndProcedure
    \end{algorithmic}
\end{algorithm}

\subsection{Recovery of the continuous state} \label{sec:continuous-state}
For control purposes, one may additionally wish to
recover the initial state corresponding to a path in $\Theta(\yb_t, \ub_t)$. For any consistent mode sequence $\md \in \Theta(\yb_t, \ub_t)$, there exists an initial state $x_{t-N_t}$ such that the linear system 
\begin{equation} \label{eq:state_estimation}
    \obs(\md) x_{t-{N_t}} = \tilde{\yb_t}(\md)
\end{equation}
with $\tilde{\yb_t}(\md) = \yb_t - G(\md)\ub_t$ is 
consistent. Of course, the solution of this system is unique 
if and only if $\rank(\obs(\md)) = n$. If this is the 
case for any path $\md$, then it is said that the system is \emph{pathwise observable} \cite{babaali_PathwiseObservabilityControllability_2003}. However, this may require a larger number of measurements than the requirement that $I_{t}^{\nc} \neq \emptyset$.
To account for this, one may simply add the condition that 
$\rank(\obs(\md)) = n$ for all $\theta \in \Theta_t$ in step~\ref{step:MO-check} of \Cref{alg:mode_estimation}. Provided that the system is \emph{pathwise observable} 
at some index $N_{\text{p}}$, then a modified version of \Cref{lem:condition-not-MO} can be shown analogously to establish finiteness 
of the window size $N_t$. However, even under this assumption, it is not 
clear that the solutions of \eqref{eq:state_estimation} for different
paths will always coincide. in \Cref{sec:design}, we 
discuss several alternative approaches for circumventing this indeterminism of the state in the controller design.

\section{Distributionally robust control} \label{sec:control}
We will now illustrate how the proposed mode-observation scheme 
can be utilized for controller design.
Our objective is to construct a stabilizing controller 
for system \eqref{eq:system}, given only the information obtained 
by the observer described above.
To this end, we adopt a distributionally robust approach,
in which the available data is used to construct a 
convex set of distributions which contains the true distribution with 
high probability.

For simplicity, we will restrict our attention here 
to the synthesis of stabilizing linear controllers, essentially 
extending the work in
\cite{schuurmans_SafeLearningBasedControl_2019} to the partially observed case.
Although interesting in its own right,
this problem will prove particularly useful as a tool for the design of stabilizing model predictive control schemes \cite{rawlings_ModelPredictiveControl_2017}. 
For instance, the controller described below could be used in conjunction with the methodology of \cite{bernardini_StabilizingModelPredictive_2012} 
(with minor modifications), extending it to a distributionally robust variant.

\subsection{Data-driven ambiguity sets} \label{sec:ambiguity}

In order to account for misestimations of transition probabilities due to finite sample sizes, we replace empirical point estimates by so-called ambiguity sets.
Conceptually, an ambiguity set is a (data-dependent) set of probability distributions that contains the true data-generating distribution with 
some specified confidence level.
Let $\set{D}_t \dfn \{(\md_k, \md_{k+1})\}_{k=0}^{t}$ denote the set of mode switches observed up to time $t$, using \Cref{alg:mode_estimation}. 
Partitioning this set into sets $\set{D}_{t,i} \dfn \{ \md_{k+1} \mid (\md_{k}, \md_{k+1}) \in \set{D}_t, \md_k = i \}$, $i \in \W$, containing the 
observed mode transitions originating in mode $i$, we obtain \iac{i.i.d.} sample 
from the probability vector $\row{\transmat}{i}$ (i.e., the $i$th row 
of the (unknown) transition matrix $\transmat$).
Let $\beta_t \in [0,1)$, $t \in \N$ denote a given (summable) sequence of desired confidence levels. 
Following the approach of \cite{schuurmans_LearningBasedDistributionallyRobust_2020a},
we define the ambiguity set as 
\begin{equation} \label{eq:ambiguity-set}
    \amb_{t,i} = \amb(\set{D}_{t,i}) \dfn \{ p \in \simplex_{\nModes} \mid \nrm{p - \hat{p}_{t,i}}_{1} \leq r_{t,i}\},
\end{equation}
where 
\(
    \hat{p}_{t,i} = \tfrac{1}{|\set{D}_{t,i}|} \sum_{\md \in \set{D}_{t,i}} \delta(\md)
\) is the empirical distribution over the set $\set{D}_{t,i}$, 
and the radius
\begin{equation} \label{eq:TV-radius}
    r_{t,i} = 
      \sqrt{\tfrac{2(\nModes \log 2 - \log \beta_t)}{|\set{D}_{t,i}| }},
\end{equation}
is chosen using standard concentration inequalities to ensure that \cite[Thm. A.6.6]{vaart_WeakConvergenceEmpirical_2000}
\begin{equation} \label{eq:ambiguity-inclusion}
    \prob[ \row{\transmat}{i} \in \amb(\set{D}_{t,i}) ] \geq 1 - \beta_t,
\end{equation}
for all $i \in \W$. Note that the probability here is taken with respect to the data set $\set{D}_{t,i}$.
We emphasize that the quantities $\hat{p}_{t,i}$ and $r_{t,i}$ can 
easily be updated online, without requiring explicit storage of the
dataset $\set{D}_{t,i}$.

Finally, it is worthwhile to mention that several other classes of ambiguity sets have been proposed in the literature (see, for instance, \cite{rahimian_DistributionallyRobustOptimization_2019} for an overview). For our purposes, however, the $\ell_1$-ambiguity set defined as \eqref{eq:ambiguity-set} is particularly suitable due to its polytopic 
structure, as we will discuss in the next section.

\subsection{Controller design} \label{sec:design}
As the mode at time $t$ is not necessarily uniquely defined,
even under pathwise observability \cite{babaali_PathwiseObservabilityControllability_2003},
the continuous state at time $t$ may not be uniquely recoverable.
Several approaches have been proposed to circumvent this issue.
For instance, in \cite{alessandri_LuenbergerObserversSwitching_2007}, path-dependent Luenberger-type observers are designed,
together with conditions under which the error dynamics are stable for any consistent path at time $t$.
However, these developments were made assuming no stochastic structure in the switching behavior.
By contrast, our goal is to use the constructed ambiguity sets \eqref{eq:ambiguity-set} to incrementally improve the controller using the observed data. 
This without jeopardizing fundamental system-theoretic properties of the closed-loop system.
To this end, we employ a more direct approach based on linear output feedback.

More specifically, 
Our goal is to construct a (time-varying) control law $u_t = K_{t} y_t$ 
such that the closed-loop system
\begin{equation} \label{eq:closed-loop}
    x_{t+1} = (A_{\md_t} + B_{\md_t}K_t C_{\md_t}) x_{t},
\end{equation}
is stable in the \emph{mean-square} sense.
System \eqref{eq:closed-loop} is said to be mean-square stable if $\lim_{t \to \infty} \E[x_t \trans{x_t}] = 0$ \cite{costa_DiscretetimeMarkovJump_2005}. 
A sufficient condition for this is the existence of a mode-dependent matrix $V_i \succ 0$, $i \in \W$ such that the Lyapunov-type condition
\begin{equation} \label{eq:lyapunov}
    \sum_{j \in \W} \transmat_{i,j} \trans{(A_{j}+ B_{j}K_t C_j)} V_{j} (A_{j} + B_{j} K_t C_j) - V_i < 0,
\end{equation}
holds for all $i$, and for all $t \in \N$ \cite{costa_DiscretetimeMarkovJump_2005}. 
The main challenge here is that $K_t$ is multiplied both on the 
left and the right, which makes it impossible to use standard
techniques to reformulate \eqref{eq:lyapunov} as \iac{LMI} \cite{schuurmans_SafeLearningBasedControl_2019,bernardini_StabilizingModelPredictive_2012,kothare_RobustConstrainedModel_1996}.

This issue was addressed in \cite{shu_StaticOutputfeedbackStabilization_2010} for 
the case where the transition matrix is known. 
Exploiting the fact that the ambiguity sets $\amb_{t,i}$ 
are polytopic, i.e., they can be written as 
$\amb_{t,i} = \conv\{q_{t,i}^{l}\}_{l=1}^{n_{\amb_{t,i}}}$, we can straightforwardly extend this methodology to the 
distributionally robust case, leading to the following result. 
\begin{thm}
    Suppose that at every time $t$,
    for sufficiently large $\alpha>0$, and any positive number $c>0$, there exists a solution $\gamma^{\star} < 0$ to
    \begin{subequations} \label{eq:LMI}
        \begin{align}
          \gamma^{\star}{=}&\minimize_{\stackrel{\gamma, V_{1,i}, V_{2,i}, V_4, L}{Q \succ0, H_{j,i}, G_{j,i}}} &&\gamma \\
          &\stt && \smallmat{V_{1,i} &\trans{V_{2,i}}\\ 
                             V_{2,i} & V_{4}} \succ 0\\
          &&& O_{i}(q^{l}_{t,i}, M, \alpha) \prec \gamma I, \label{eq:condition-omega}\\
          &&& \gamma \geq -c, \forall  l \in \natseq{1}{n_{\amb_{t,i}}},
        \end{align}
    \end{subequations}
    for all $j \in \{1,2\}$ and $i \in \W$, with $O_{i}$ as defined in \eqref{eq:control-design-matrix} (see \Cref{sec:LMI}), where 
    $M = (M_{i})_{i \in \W}$ is a given (potentially mode-dependent) stabilizing state-feedback gain with respect to all $\row{\transmat}{i} \in \amb_{t,i}, i \in \W$ and $\{q_{t,i}^{l}\}_{l=1}^{n_{\amb_{t,i}}}$ are the vertices of 
    the polytopic ambiguity set $\amb_{t,i}$.
    Let $K_t = Q^{-1}L$ denote the corresponding output feedback gain.  
    If furthermore, the confidence levels $\beta_t$ are chosen to satisfy
    $\sum_{t=0}^{\infty} \beta_{t} < \infty$, 
    then $u_t = K_t y_t$ is a mean-square stabilizing controller.
\end{thm}
\begin{proof}
    If $\row{\transmat}{i} \in \amb_{t,i}$ for all $i \in \W$, then, 
    since $O_{i}(\argdot, M, \alpha)$ is affine, \cite[Thm. 1]{shu_StaticOutputfeedbackStabilization_2010} implies that at time 
    $t$ \eqref{eq:lyapunov} holds for all $i \in \W$.
    Now let $E_t \dfn \{ \exists i \in \W : \row{\transmat}{i} \notin \amb_{t,i}\}$ denote the event where at least one row of the transition matrix does not lie in the corresponding ambiguity set at time $t$.
    By \eqref{eq:ambiguity-inclusion} and the union bound, we have that $\prob[E_t] \leq \nModes \beta_t$.
    Since $\beta_t$ was assumed to be summable, the Borel-Cantelli lemma \cite[Thm. 4.3]{billingsley_ProbabilityMeasure_1995} implies that 
    with probability 1, there exists a finite time $T$, such that the event $E_t$ occurs for all $t > T$. Thus, the Lyapunov condition \eqref{eq:lyapunov} holds for all $i \in \W$ and for all sufficiently large $t$. 
\end{proof}
The stabilizing state-feedback gain $M$ can be computed
using standard techniques, often yielding a separate \ac{LMI} (e.g., \cite{bernardini_StabilizingModelPredictive_2012,schuurmans_SafeLearningBasedControl_2019}).
For the confidence levels we choose $\beta_{t} = 0.5 (t+1)^{-2}$ in our experiments, as this sequence is summable, but decreases sufficiently slowly to ensure that the radii $r_{t,i}$ will converge to 0.  
Indeed, since by assumption, the Markov chain is ergodic,
\cite[Lem. 6]{wolfer_MinimaxLearningErgodic_2019} essentially states that the number of visits in every mode (i.e., $|\mathcal{D}_{t,i}|$) asymptotically grows linearly with $t$, so that by \eqref{eq:TV-radius}, $\lim_{t \to \infty} r_{t,i} = \tfrac{\log t}{t} = 0$.
As a result, we may expect subsequent refinements of $K_t$ to approach 
the control gain obtained when the transition matrix $\transmat$ is known.

\section{Illustrative examples} \label{sec:examples}
\subsection{Mode estimation}
Consider the system \eqref{eq:system} with 
\(
      A_1 = \smallmat{0.45 & 0\\ 0 & 0.4}, 
      A_2 = \smallmat{0.25 &-0.20 \\ 0.04 & 0.4},
      B_1 = B_2 = \smallmat{0.3\\ 0.4}, 
      C_1 = C_2 = \smallmat{2&1}, 
\)
driven by random input $u_t$, drawn \ac{i.i.d.} from a Gaussian distribution, i.e., $u_t \sim \mathcal{N}(0,I_{\na})$ as described in \cite{ragot_SwitchingTimeEstimation_2003}.
This system is not weakly \ac{MO} at least up to index $N=15$, in the sense of \Cref{lem:controlled-discernibility}. However,
running \Cref{alg:mode_estimation} with $\nc=2$ for 5000 randomly drawn states $x_0 \sim \mathcal{N}(0,I_{\ns})$, and with uniform transition probabilities (i.e., $\transmat_{i,j} = 0.5$, $i,j\in\W$) we find that in all cases, $N_t$ converges to $3$, with $|\Theta_t| \leq 4$ for all $t$. 

Similarly, the system with $A_i$, $B_i$ and $C_i$ as in \cite[sec. V]{bako_NewStateObserver_2011} (omitted here for sake of space) is $(4,1,1)$-MO, and thus weakly MO in the sense of \Cref{lem:controlled-discernibility}. Yet, in practice, 
we find that in almost all cases, \Cref{alg:mode_estimation} with $\nc =2$ converges to $N_t = 3$, with the path uniquely identified, i.e., $|\Theta_t| = 1$ for all $t \geq 1$.
It is clear that manual selection of the observation window size based on 
offline observability computations will -- when feasible -- tend to lead to a strict overestimation of the practically required window length.

\subsection{Distributionally robust control} \label{sec:dr-control-ex}

Consider the (autonomously instable) system \eqref{eq:system} with
\[
\begin{aligned}
  A_1 &= \smallmat{1.05& 1.8\\0 &1.1}, A_2 = \smallmat{0.95 & 0.7\\ 0 & 0.95},\\
  B_1 &= \smallmat{0.9& 0 \\ 0&0}, B_2 = \smallmat{0.8 & 0 \\ 0& 1.4},
  C_1 = \smallmat{1 & 1 \\ 0 & 0}, C_2 = \smallmat{1&0\\0&1},
\end{aligned}
\]
and with uniform transition matrix: $\bar{\transmat}_{i,j} = 0.5$, $i,j \in \W$.

We compare three controllers: 
\begin{inlinelist}
    \item \label{item:rob}A \textbf{robust} controller, assuming no knowledge of the 
    transition probabilities, computed by solving \eqref{eq:LMI}
    with $\amb_{t,i} = \simplex_{2}^2$;
    \item \label{item:stoch} A \textbf{stochastic} controller, which has access to the 
    transition matrix $\bar{\transmat}$, and is obtained by solving the \eqref{eq:LMI} with $\amb_{t,i} = \{\row{\bar{\transmat}}{i}\}$; and
    \item a \textbf{distributionally robust} (DR) controller, using the ambiguity set described in \Cref{sec:ambiguity}. 
\end{inlinelist}
We perform a closed-loop simulation of the three controllers for a
random realization of the mode sequence from a fixed initial state.
Initially, the DR controller behaves indistinguishably from the 
robust controller, as no data is obtained yet and the ambiguity set coincides with the probability simplex. After 50 steps, we apply a sudden additive disturbance to the system.
Note that such a disturbance can be detected in the mode estimation scheme, as at this time, $\Theta_t$ will be empty.
In our implementation, we handle this case by discarding data within the observation window and resetting $N_t$ back to 0.
By this time, the ambiguity set radius has already decreased considerably and as illustrated in \Cref{fig:state_vs_time},
the DR controller shows faster disturbance rejection than the robust approach. 
As more data is observed and the ambiguity sets $\amb_{t,i} \to \{\bar{\transmat_{i}}\}$, the DR controller converges to the stochastic controller. 
\begin{figure}[ht!]
   \centering
   \includegraphics[]{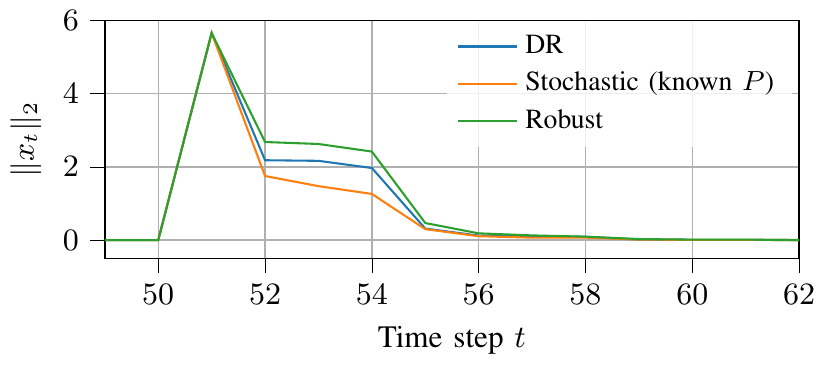}
   \caption{Evolution of the state norm in the example of \Cref{sec:dr-control-ex}.}
   \label{fig:state_vs_time}
\end{figure}

\section{Conclusion}
We have proposed an adaptive mode-observation scheme 
for Markov jump linear systems, which does not require 
explicit determination of the index of observability of the 
system. We have illustrated, based on a simple 
linear output feedback control design, how this mode observer can be used to construct ambiguity sets over the mode transition probabilities, 
allowing to improve control performance over time. 
In future work, we aim to derive more efficient formulations for the computation of stabilizing output feedback gains, akin to the ones proposed in \cite{schuurmans_SafeLearningBasedControl_2019}. Furthermore, we 
plan to integrate the proposed mode observer with 
distributionally robust model predictive control schemes \cite{schuurmans_LearningBasedDistributionallyRobust_2020a}.

\bibliographystyle{ieeetr}
\bibliography{references}

\begin{thebibliography}{10}

\bibitem{fischer_OptimalSequencebasedLQG_2013}
J.~Fischer, A.~Hekler, M.~Dolgov, and U.~D. Hanebeck, ``Optimal sequence-based
  {{LQG}} control over {{TCP}}-like networks subject to random transmission
  delays and packet losses,'' in {\em 2013 {{American Control Conference}}},
  pp.~1543--1549, June 2013.

\bibitem{schuurmans_LearningBasedRiskAverseModel_2020a}
M.~Schuurmans, A.~Katriniok, H.~E. Tseng, and P.~Patrinos, ``Learning-{{Based
  Risk}}-{{Averse Model Predictive Control}} for {{Adaptive Cruise Control}}
  with {{Stochastic Driver Models}},'' in {\em {{IFAC}} 2020 {{World
  Congress}}}, ({Berlin}), pp.~15337--15342, 2020.

\bibitem{blair_ContinuousTimeRegulationClass_1975}
W.~P. Blair and D.~D. Sworder, ``Continuous-{{Time Regulation}} of a {{Class}}
  of {{Econometric Models}},'' {\em IEEE Transactions on Systems, Man, and
  Cybernetics}, vol.~SMC-5, pp.~341--346, May 1975.

\bibitem{rahimian_DistributionallyRobustOptimization_2019}
H.~Rahimian and S.~Mehrotra, ``Distributionally {{Robust Optimization}}: {{A
  Review}},'' {\em arXiv:1908.05659 [cs, math, stat]}, Aug. 2019.

\bibitem{coppens_DatadrivenDistributionallyRobust_2020}
P.~Coppens, M.~Schuurmans, and P.~Patrinos, ``Data-driven distributionally
  robust {{LQR}} with multiplicative noise,'' in {\em Learning for {{Dynamics}}
  and {{Control}}}, pp.~521--530, {PMLR}, July 2020.

\bibitem{xu_DistributionallyRobustMarkov_2010}
H.~Xu and S.~Mannor, ``Distributionally {{Robust Markov Decision Processes}},''
  in {\em Advances in {{Neural Information Processing Systems}} 23} (J.~D.
  Lafferty, C.~K.~I. Williams, J.~{Shawe-Taylor}, R.~S. Zemel, and A.~Culotta,
  eds.), pp.~2505--2513, {Curran Associates, Inc.}, 2010.

\bibitem{schuurmans_LearningBasedDistributionallyRobust_2020a}
M.~Schuurmans and P.~Patrinos, ``Learning-{{Based Distributionally Robust Model
  Predictive Control}} of {{Markovian Switching Systems}} with {{Guaranteed
  Stability}} and {{Recursive Feasibility}},'' in {\em 59th {{IEEE Conference}}
  on {{Decision}} and {{Control}}, {{CDC}} 2020}, Proceedings of the 59th
  {{IEEE Conference}} on {{Decision}} and {{Control}} ({{CDC}}), ({Jeju Island,
  South Korea}), pp.~4287--4292, Dec. 2020.

\bibitem{chow_FrameworkTimeconsistentRiskaverse_2014}
Y.-L. Chow and M.~Pavone, ``A framework for time-consistent, risk-averse model
  predictive control: {{Theory}} and algorithms,'' in {\em 2014 {{American
  Control Conference}}}, ({Portland, OR, USA}), pp.~4204--4211, {IEEE}, June
  2014.

\bibitem{patrinos_StochasticModelPredictive_2014}
P.~Patrinos, P.~Sopasakis, H.~Sarimveis, and A.~Bemporad, ``Stochastic model
  predictive control for constrained discrete-time {{Markovian}} switching
  systems,'' {\em Automatica}, vol.~50, pp.~2504--2514, Oct. 2014.

\bibitem{costa_DiscretetimeMarkovJump_2005}
O.~L. d.~V. Costa, M.~D. Fragoso, and R.~P. Marques, {\em Discrete-Time
  {{Markov}} Jump Linear Systems}.
\newblock Probability and Its Applications, {London}: {Springer}, 2005.

\bibitem{ackerson_StateEstimationSwitching_1970}
G.~Ackerson and K.~Fu, ``On state estimation in switching environments,'' {\em
  IEEE Transactions on Automatic Control}, vol.~15, pp.~10--17, Feb. 1970.

\bibitem{tugnait_DetectionestimationSchemeState_1979}
J.~Tugnait and A.~Haddad, ``A detection-estimation scheme for state estimation
  in switching environments,'' {\em Automatica}, vol.~15, pp.~477--481, July
  1979.

\bibitem{mazor_InteractingMultipleModel_1998}
E.~Mazor, A.~Averbuch, Y.~{Bar-Shalom}, and J.~Dayan, ``Interacting multiple
  model methods in target tracking: A survey,'' {\em IEEE Transactions on
  Aerospace and Electronic Systems}, vol.~34, pp.~103--123, Jan. 1998.

\bibitem{alriksson_ObserverSynthesisSwitched_2006}
P.~Alriksson and A.~Rantzer, ``Observer {{Synthesis}} for {{Switched
  Discrete}}-{{Time Linear Systems}} using {{Relaxed Dynamic Programming}},''
  in {\em 17th {{International Symposium}} on {{Mathematical Theory}} of
  {{Networks}} and {{Systems}}, 2006}, 2006.

\bibitem{bako_NewStateObserver_2011}
L.~Bako and S.~Lecoeuche, ``A new state observer for switched linear systems
  using a non-smooth optimization approach,'' in {\em 2011 50th {{IEEE
  Conference}} on {{Decision}} and {{Control}} and {{European Control
  Conference}}}, pp.~5329--5336, Dec. 2011.

\bibitem{ji_ControllabilityObservabilityDiscretetime_1988}
Y.~Ji and H.~J. Chizeck, ``Controllability, observability and discrete-time
  markovian jump linear quadratic control,'' {\em International Journal of
  Control}, vol.~48, pp.~481--498, Aug. 1988.

\bibitem{vidal_ObservabilityIdentifiabilityJump_2002}
R.~Vidal, A.~Chiuso, and S.~Soatto, ``Observability and identifiability of jump
  linear systems,'' in {\em Proceedings of the 41st {{IEEE Conference}} on
  {{Decision}} and {{Control}}, 2002.}, vol.~4, pp.~3614--3619 vol.4, Dec.
  2002.

\bibitem{elhamifar_RankTestsObservability_2009}
E.~Elhamifar, M.~Petreczky, and R.~Vidal, ``Rank tests for the observability of
  discrete-time jump linear systems with inputs,'' in {\em 2009 {{American
  Control Conference}}}, pp.~3025--3032, June 2009.

\bibitem{babaali_PathwiseObservabilityControllability_2003}
M.~Babaali and M.~Egerstedt, ``Pathwise observability and controllability are
  decidable,'' in {\em 42nd {{IEEE International Conference}} on {{Decision}}
  and {{Control}} ({{IEEE Cat}}. {{No}}.{{03CH37475}})}, vol.~6, pp.~5771--5776
  Vol.6, Dec. 2003.

\bibitem{babaali_ObservabilitySwitchedLinear_2004}
M.~Babaali and M.~Egerstedt, ``Observability of {{Switched Linear Systems}},''
  in {\em Hybrid {{Systems}}: {{Computation}} and {{Control}}} (R.~Alur and
  G.~J. Pappas, eds.), Lecture {{Notes}} in {{Computer Science}}, ({Berlin,
  Heidelberg}), pp.~48--63, {Springer}, 2004.

\bibitem{babaali_AsymptoticObserversDiscreteTime_2005}
M.~Babaali and M.~Egerstedt, ``Asymptotic {{Observers}} for {{Discrete}}-{{Time
  Switched Linear Systems}},'' {\em IFAC Proceedings Volumes}, vol.~38, no.~1,
  pp.~1269--1274, 2005.

\bibitem{baglietto_ActiveModeObservability_2007}
M.~Baglietto, G.~Battistelli, and L.~Scardovi, ``Active mode observability of
  switching linear systems,'' {\em Automatica}, vol.~43, pp.~1442--1449, Aug.
  2007.

\bibitem{baglietto_ActiveModeObservation_2009}
M.~Baglietto, G.~Battistelli, and L.~Scardovi, ``Active mode observation of
  switching systems based on set-valued estimation of the continuous state,''
  {\em International Journal of Robust and Nonlinear Control}, vol.~19, no.~14,
  pp.~1521--1540, 2009.

\bibitem{alessandri_RecedinghorizonEstimationSwitching_2005}
A.~Alessandri, M.~Baglietto, and G.~Battistelli, ``Receding-horizon estimation
  for switching discrete-time linear systems,'' {\em IEEE Transactions on
  Automatic Control}, vol.~50, pp.~1736--1748, Nov. 2005.

\bibitem{baglietto_ModeobservabilityDegreeDiscretetime_2014}
M.~Baglietto, G.~Battistelli, and P.~Tesi, ``Mode-observability degree in
  discrete-time switching linear systems,'' {\em Systems \& Control Letters},
  vol.~70, pp.~69--76, Aug. 2014.

\bibitem{billingsley_ProbabilityMeasure_1995}
P.~Billingsley, {\em Probability and Measure}.
\newblock Wiley Series in Probability and Mathematical Statistics, {New York}:
  {Wiley}, 3rd ed~ed., 1995.

\bibitem{halimi_ModelBasedModesDetection_2015}
M.~Halimi, G.~Mill{\'e}rioux, and J.~Daafouz, ``Model-{{Based Modes Detection}}
  and {{Discernibility}} for {{Switched Affine Discrete}}-{{Time Systems}},''
  {\em IEEE Transactions on Automatic Control}, vol.~60, pp.~1501--1514, June
  2015.

\bibitem{schuurmans_SafeLearningBasedControl_2019}
M.~Schuurmans, P.~Sopasakis, and P.~Patrinos, ``Safe {{Learning}}-{{Based
  Control}} of {{Stochastic Jump Linear Systems}}: A {{Distributionally Robust
  Approach}},'' in {\em 2019 {{IEEE}} 58th {{Conference}} on {{Decision}} and
  {{Control}} ({{CDC}})}, pp.~6498--6503, Dec. 2019.

\bibitem{rawlings_ModelPredictiveControl_2017}
J.~B. Rawlings, D.~Q. Mayne, and M.~M. Diehl, {\em Model Predictive Control:
  Theory, Computation, and Design}.
\newblock {Madison, Wisconsin}: {Nob Hill Publishing}, 2nd edition~ed., 2017.

\bibitem{bernardini_StabilizingModelPredictive_2012}
D.~Bernardini and A.~Bemporad, ``Stabilizing {{Model Predictive Control}} of
  {{Stochastic Constrained Linear Systems}},'' {\em IEEE Transactions on
  Automatic Control}, vol.~57, pp.~1468--1480, June 2012.

\bibitem{vaart_WeakConvergenceEmpirical_2000}
A.~W. van~der Vaart and J.~A. Wellner, {\em Weak Convergence and Empirical
  Processes: With Applications to Statistics}.
\newblock {New York}: {Springer}, 2000.

\bibitem{alessandri_LuenbergerObserversSwitching_2007}
A.~Alessandri, M.~Baglietto, and G.~Battistelli, ``Luenberger observers for
  switching discrete-time linear systems,'' {\em International Journal of
  Control}, vol.~80, pp.~1931--1943, Dec. 2007.

\bibitem{kothare_RobustConstrainedModel_1996}
M.~V. Kothare, V.~Balakrishnan, and M.~Morari, ``Robust constrained model
  predictive control using linear matrix inequalities,'' {\em Automatica},
  vol.~32, pp.~1361--1379, Oct. 1996.

\bibitem{shu_StaticOutputfeedbackStabilization_2010}
Z.~Shu, J.~Lam, and J.~Xiong, ``Static output-feedback stabilization of
  discrete-time {{Markovian}} jump linear systems: {{A}} system augmentation
  approach,'' {\em Automatica}, vol.~46, pp.~687--694, Apr. 2010.

\bibitem{wolfer_MinimaxLearningErgodic_2019}
G.~Wolfer and A.~Kontorovich, ``Minimax {{Learning}} of {{Ergodic Markov
  Chains}},'' in {\em Algorithmic {{Learning Theory}}}, pp.~903--929, Mar.
  2019.

\bibitem{ragot_SwitchingTimeEstimation_2003}
J.~Ragot, D.~Maquin, and E.~A. Domlan, ``Switching {{Time Estimation}} of
  {{Piecewise Linear Systems}}. {{Application}} to {{Diagnosis}},'' {\em IFAC
  Proceedings Volumes}, vol.~36, pp.~651--656, June 2003.

\end{thebibliography}
\ifArxiv
\begin{appendix}
\subsection{Technical lemmas and auxiliary results}
Let us define
$$
    \hat{c}(\md, \md')
    {}\dfn{}
    \{x \in \Re^{\ns} \mid \exists x' \in \Re^{\ns} : C_{\md} x = C_{\md'}x' \},
$$
and
$$
\begin{aligned}
    c(\md, \md', u)
    {}&\dfn{}
    \left\{
    x \in \Re^{\ns}
        \sep
    \begin{matrix}
        \obsaug(\md, x, u) = \obsaug(\md', x', u),\\
        x' \in \Re^{\ns}
    \end{matrix}
    \right \}.
\end{aligned}
$$
Clearly, the system is $(N, \alpha, \omega)$-\ac{MO} if and only 
if $c(\md, \md', u) = \emptyset$ for all $\md, \md' \in W^{N+1}$ with  $\seq{\md}{\alpha}{\omega} \neq \seq{\md'}{\alpha}{\omega}$ and 
almost every $u \in \Re^{\na N}$.

The following facts will be useful in proving the results 
from the main text. 
First, notice that 
\begin{equation} \label{eq:subset}
\begin{aligned}
  c(\md, \md', u)
      {}&={}
      \left\{
      x \in \Re^{\ns}
          \sep
      \begin{matrix}
          C_{\md_0}x = C_{\md_0'}x',\; x' \in \Re^{\ns},\\
           \obsaug(\seq{\md}{1}{N}, \bar{x}, \seq{u}{1}{N-1})\\
          {}={}
          \obsaug(\seq{\md}{1}{N}', \bar{x}', \seq{u}{1}{N-1}),\\
          \bar{x}  = A_{\md_0}x,\;
          \bar{x}' = A_{\md_0'} x'
      \end{matrix}
  \right\}\\
  {}&\subseteq{}
  \hat{c}(\md_{0}, \md_{0}') \cap c(\seq{\md}{1}{N}, \seq{\md'}{1}{N}, \seq{u}{1}{N-1}). 
\end{aligned}
\end{equation}

\begin{lem} \label{lem:inclusion_conflict_1}
Let $\cat{\md \bar{\md}}, \cat{\md' \bar{\md}'} \in \W^{N+1}$ 
denote the concatenation of the paths $\md, \md' \in \W^{N}$ and 
the modes $\bar{\md}, \bar{\md}' \in \W$, respectively. 
It holds that 
$$
    c(\cat{\md \bar{\md}}, \cat{\md' \bar{\md}'}, u) \subseteq c(\md, \md', \seq{u}{0}{N-2})
$$
for all $u = (u_0, \dots, u_{N-1})\in \Re^{\na N}$.
\end{lem}
\begin{proof}
The statement follows directly from the definition: 
$$
\begin{aligned}
      &x \in c(\cat{\md \bar{\md}}, \cat{\md' \bar{\md}'}, u) \subseteq \Re^{\ns}\\
      &{}\Leftrightarrow{}
        \obsaug(\cat{\md \bar{\md}}, x, u) = \obsaug(\cat{\md' \bar{\md}'}, x', u)
        \text{ for some } 
        x' \in \Re^{\ns}\\
        &{}\Leftrightarrow{}
    \begin{cases} 
        \obsaug(\md, x, \seq{u}{0}{N-2}) = \obsaug(\md', x', \seq{u}{0}{N-2})\\
        C_{\bar{\md}} x_{N} = C_{\bar{\md}'} x_{N}',
    \end{cases}
\end{aligned}
$$
with $x_{N}$ the state at time $N$, obtained by applying \eqref{eq:dynamics} from initial condition $x$, under mode sequence $\md$ and controls $\seq{u}{0}{N-2}$. $x_{N}'$ is analogously defined with initial condition $x'$ and mode sequence $\md'$.
On the other hand, $x \in c(\md, \md', \seq{u}{0}{N-2})$ if and only if  
$$  
\obsaug(\md, x, \seq{u}{0}{N-2}) = \obsaug(\md', x', \seq{u}{0}{N-2}) \text{ for some } x' \in \Re^{\ns}.
$$ 
Therefore, if $x \in c(\cat{\md \bar{\md}}, \cat{\md' \bar{\md}'}, u)$, then $x \in c(\md, \md', \seq{u}{0}{N-2})$, as required.
\end{proof}

\begin{lem} \label{lem:inverse-transform-inclusion}
Let $\cat{\bar{\md}\md}, \cat{\bar{\md}'\md'} \in \W^{N+1}$ denote the concatenations 
of $\bar{\md}, \bar{\md}' \in \W$ and $\md, \md' \in \W^{N}$, respectively.
Then,
$$
    c(\cat{\bar{\md}\md},  \cat{\bar{\md}'\md'}, u) \subseteq T_{\bar{\md}, u_0}^{-1}c(\md, \md', \seq{u}{1}{N-1}), 
$$
for any $u = (u_0, \dots, u_{N-1}) \in \Re^{\na N}$, where $T_{\bar{\md},u_0}^{-1}(C)$ denotes the pre-image of a set $C$
under the affine map $T_{\bar{\md},u_0}: x \mapsto A_{\bar{\md}}x + B_{\bar{\md}}u_{0}$.
\end{lem}
\begin{proof}
This again follows directly from the definition:
$$
\begin{aligned}
&c(\cat{\bar{\md} \md}, \cat{\bar{\md}' \md'}, u)\\
    &{}={}
        \left \{x \in \Re^{\ns} \sep
        \begin{matrix}
            C_{\bar{\md}} x = C_{\bar{\md}'} x', \; x' \in \Re^{\ns},\\
            \obsaug(\md, \bar{x}, \seq{u}{1}{N-1}) = \obsaug(\md', \bar{x}', \seq{u}{1}{N-1}),\\
            \bar{x} = T_{\bar{\md}, u_0}(x), \quad \bar{x}' = T_{\bar{\md}', u_{0}}(x')
        \end{matrix}
        \right \}\\
    &{}\subseteq{}
        \left \{x \in \Re^{\ns} \sep
        \begin{matrix}
            \obsaug(\md, \bar{x}, \seq{u}{1}{N-1}) = \obsaug(\md', \bar{x}', \seq{u}{1}{N-1})\\
            \bar{x} =  T_{\bar{\md}, u_0}(x), \, \bar{x}' \in \col(A_{\bar{\md}'}){+}B_{\bar{\md}'}u_0{+}p_{\bar{\md}'}
        \end{matrix}
        \right \} \\ 
    &{}\subseteq{}
        \left \{x \in \Re^{\ns} \sep
        \begin{matrix}
            \obsaug(\md, \bar{x}, \seq{u}{1}{N-1}) = \obsaug(\md', \bar{x}', \seq{u}{1}{N-1})\\
            \bar{x} = T_{\bar{\md}, u_0}(x), \quad \bar{x}' \in \Re^{\ns}
        \end{matrix}
        \right \} \\ 
    &{}={}
        T_{\bar{\md}, u_0}^{-1}c(\md, \md', \seq{u}{1}{N-1}).
\end{aligned}
$$
\end{proof}

\subsection{Deferred proofs}

\begin{appendixproof}{lem:incremental-update}
    Let $x_0, \dots, x_{N+1}$ denote the true state trajectory 
    satisfying dynamics \eqref{eq:dynamics} under the control sequence 
    $\seq{u}{0}{N}$, true mode sequence $\seq{\md^{\star}}{0}{N+1}$, yielding measurements $\seq{y}{0}{N+1}$. 
    Consider $\md' = \cat{\md \bar{\md}} \in \Theta(\seq{y}{1}{N+1}, \seq{u}{1}{N}) \subseteq \W^{N+1}$ as defined in the statement. By construction, we have that 
    \begin{inlinelist}
        \item $\seq{y}{1}{N+1} = \obsaug(\md', x_1, \seq{u}{1}{N})$; and 
        \item $x_1 = A_{\md^\star_0} x_0 + B_{\md^\star_0}u_0$, $y_0 = C_{\md_0^\star} x_0$. 
    \end{inlinelist}
    Hence, taking $\underline \md = \md^\star_0$ and writing the measurement equation \eqref{eq:measurement-equation} explicitly, it is immediate that $\seq{y}{0}{N+1} \in \obsaug(\cat{\underline \md \md \bar{\md}}, x_0, \seq{u}{0}{N})$,
    and thus $\cat{\underline \md \md} \in \Theta(\seq{y}{0}{N}, \seq{u}{0}{N-1})$. The second part of the statement follows from a very similar argument.
\end{appendixproof}

\begin{appendixproof}{prop:extension-of-MO-index}
    Let $\md \neq \md'$ be two different paths of length $N'+ 1$ with
    $N' \dfn N-\alpha-\omega$
    and let $\underline{\lambda}, \underline{\lambda}' \in \W^{\alpha}$
    be any two paths of length $\alpha$ and $\bar{\lambda}, \bar{\lambda'} \in \W^{\omega}$
    be any two paths of length $\omega$.
    Let $\pi \dfn \cat{\underline{\lambda} \md \bar{\lambda}}, \pi'\dfn \cat{\underline{\lambda'} \md' \bar{\lambda'}} \in \W^{N+1}$ denote the concatenations of these paths.
    It suffices to show that these paths are mutually discernable.
    Note that since $\md \neq \md'$, at least one the following cases applies.
    \begin{enumerate}[leftmargin=0cm,itemindent=.5cm,labelwidth=\itemindent ,labelsep=0pt, align=left, label=\textbf{Case \arabic* }]
    \item
    ($\seq{\md}{0}{N'-1} \neq \seq{\md}{0}{N'-1}'$): $\pi$ and $\pi'$
    can be written as
    $\cat{\underline{\lambda} \seq{\md}{0}{N'-1} \mu}$
    and
    $\cat{\underline{\lambda'} \seq{\md'}{0}{N'-1}\mu'}$,
    respectively,
    with 
    $\mu \dfn \md_{N'} \bar{\lambda}$
    and
    $\mu' \dfn \md'_{N'} \bar{\lambda'}$.
    By assumption, we have that
    $$
        c\left(\cat{\underline{\lambda}\seq{\md}{0}{N'-1} \seq{\mu}{0}{\omega-1}}, \cat{\underline{\lambda'}\seq{\md}{0}{N'-1}' \seq{\mu'}{0}{\omega-1}}, u\right) = \emptyset,
    $$
    for every $u \in U = \Re^{\na (N-1)} \setminus \mathcal{K}$ with $\mathcal{K}$ some null set. Furthermore, by \Cref{lem:inclusion_conflict_1}, 
    $$
        \begin{aligned}
          c(\pi, \pi', u) &=  c\left(\cat{\underline{\lambda}\seq{\md}{0}{N'-1} \mu}, \cat{\underline{\lambda'}\seq{\md}{0}{N'-1}' \mu'}, u\right)  \\ 
          \subseteq c\Big(
              \cat{\underline{\lambda}&\seq{\md}{0}{N'-1} \seq{\mu}{0}{\omega-1}}, \cat{\underline{\lambda}\seq{\md}{0}{N'-1}' \seq{\mu'}{0}{\omega-1}}, \seq{u}{0}{N-1}
              \Big)
        \end{aligned}
    $$
    for all $u \in \Re^{\na N}$, we conclude that $c(\pi, \pi', u) = \emptyset$ for every $u \in U' \dfn U \times \Re^{\na}$. By basic set-theoretic arguments, $\Re^{N(\na)} \setminus U' = \mathcal{K} \times \Re^{\na}$ and is therefore null. We conclude that the paths $\pi$ and $\pi'$ are discernible.
    \item
    ($\seq{\md}{1}{N'} \neq \seq{\md}{1}{N'}'$):
    We can write
    $\seq{\pi}{1}{N} = \cat{\mu \seq{\md}{1}{N'} \bar{\lambda}}$,
    $\seq{\pi'}{1}{N} = \cat{\mu' \seq{\md'}{1}{N'} \bar{\lambda'}}$,
    with
    $\mu \dfn \cat{\seq{\underline{\lambda}}{1}{\alpha} \md_{0}}$
    and
    $\mu' \dfn \cat{\seq{\underline{\lambda'}}{1}{\alpha} \md_{0}'}$.
    Since $\mu, \mu' \in \W^{\alpha}$,
    $\seq{\md}{1}{N'} \neq \seq{\md}{1}{N'} \in \W^{N'}$
    and 
    $\bar{\lambda}, \bar{\lambda'} \in \W^{\omega}$, 
    we have that 
    $$
        c(\seq{\pi}{1}{N}, \seq{\pi'}{1}{N}, u) = \emptyset
    $$
    for almost every $u \in \Re^{\na (N-1)}$.
    Moreover, by \Cref{lem:inverse-transform-inclusion},
    $$
        \begin{aligned}
          c(\pi, \pi', u)
          {}&\subseteq{}
            T_{\pi_0,u_0}^{-1}(c(\seq{\pi}{1}{N}, \seq{\pi'}{1}{N}, \seq{u}{1}{N-1}), 
        \end{aligned}
    $$
    thus $c(\pi, \pi', u) = \emptyset$ for almost every $u \in \Re^{\na N}$, 
    as required.
    \end{enumerate}
    Since $\md \neq \md'$ were arbitrary, this concludes the proof.
\end{appendixproof}
\subsection{LMI-based controller design} \label{sec:LMI}
For completeness, we recall the control design matrix introduced in \cite[Thm. 3]{shu_StaticOutputfeedbackStabilization_2010} here. 
\begin{equation} \label{eq:control-design-matrix}
    O_{i}(q)
    {}\dfn{}\smallmat{O_{11,i}& * & * & * \\ 
    O_{21,i} & O_{22,i} & * & * \\
    O_{31,i} & O_{32,i} & O_{33,i} & * \\
    O_{41,i} & O_{42,i} & O_{43,i} & O_{44,i}}, 
\end{equation}
with 
\[
   \begin{aligned}
     O_{11,i} &= H_{1,i} A_i + \trans{(H_{1,i} A_i)} - V_{1,i}\\
     & \quad + 2 \alpha \trans{M_i} Q_4 M_i - 2 \alpha (\trans{(\trans{M_i} L C_i)}  + \trans{M_i} L C_i)\\
     O_{21,i} &= H_{2,i} A_{i} + \trans{(H_{1,i} B_i)} + 2 \alpha L C_i\\
     O_{22,i} &= H_{2,i} B_i + \trans{(H_{2,i} B_i)} - 2 \alpha Q_{4}\\
     O_{31,i} &= G_{1,i} A_i + \trans{H_{1,i}}\\
     O_{32,i} &= G_{1,i} B_i - \trans{H_{2,i}}\\
     O_{33,i} &= \tsum_{j\in\W} q_j V_{1,j} - G_{1,i} - \trans{G_{1,i}}\\
     O_{41,i} &= G_{2,i} A_i + L C_i\\
     O_{42,i} &= G_{2,i} B_i + Q_{4}\\
     O_{43,i} &= -G_{2,i} + \tsum_{j\in\W} q_j V_{2,j}\\
     O_{44,i} &= \tsum_{j \in \W} q_j V_{4,j} - 2 Q_4.
   \end{aligned}
\]
\end{appendix}
\fi
\end{document}